\documentclass[12pt]{amsart}
\usepackage{amssymb,verbatim,enumerate,ifthen}
\usepackage{cite}
\usepackage[mathscr]{eucal}
\usepackage[utf8]{inputenc}
\usepackage[T1]{fontenc}
\usepackage{esint}
\newtheorem{thm}{Theorem}[section]

\newtheorem{prop}[thm]{Proposition}
\newtheorem{lem}[thm]{Lemma}

\theoremstyle{definition}

\newtheorem{rem}[thm]{Remark}
\newtheorem{exa}[thm]{Example}

\numberwithin{equation}{section}
\def\eq#1{{\rm(\ref{#1})}}
\def\dotref#1#2{\ref{#1}.\ref{#1.#2}}
\def\Eq#1#2{\ifthenelse{\equal{#1}{*}}
  {\begin{equation*}\begin{aligned}[]#2\end{aligned}\end{equation*}}
  {\begin{equation}\begin{aligned}[]\label{#1}#2\end{aligned}\end{equation}}}

\def\B{\mathscr{B}}

\def\M{\mathscr{M}}
\def\P{\mathscr{P}}

\def\CS{\mathcal{CS}}

\def\MM{\mathcal{M}}

\newcommand{\uel}[2]{\ifthenelse{\equal{#1}{}}{#2,\dots,#2}{\underbrace{#2,\dots,#2}_{#1\text{ times}}}}
\newcommand\R{\mathbb{R}}
\def\EE{\mathcal{E}}
\newcommand\N{\mathbb{N}}

\newcommand{\abs}[1]{\left| #1 \right| }

\newcommand{\dotvec}[3][SKIPPED]{
\ifthenelse{\equal{#1}{SKIPPED}}
  {#2,\dots,#3}
  {\underbrace{#2,\dots,#3}_{#1\text{ entries}}}
}

\begin{document}
\begin{abstract}
We prove that whenever $M_1,\dots,M_n\colon I^k \to I$, (%
$n,k \in \mathbb{N}$) are symmetric, continuous means on the interval $I$ and
$S_1,\dots,S_m\colon I^k \to I$ ($m <n$) satisfies a sort of embeddability assumptions then
for every continuous function $\mu \colon I^n \to \R$ which is strictly monotone in each coordinate, the functional equation
$$
\mu(S_1(v),\dots,S_m(v),\underbrace{F(v),\dots,F(v)}_{(n-m)\text{ times}})=\mu(M_1(v),\dots,M_n(v))
$$
has the unique solution $F=F_\mu \colon I^k \to I$ which is a mean.
We deliver some sufficient conditions so that $F_\mu$ is well-defined (in particular uniquely determined) and study its properties. 

The background of this research is to provide a broad overview of the family of Beta-type means introduced in (Himmel and Matkowski,~2018).
\end{abstract}
\title[Pexider invariance equation for embeddable mappings]{Pexider invariance equation for embeddable mean-type mappings}
\author[P. Pasteczka]{Pawe\l{} Pasteczka}
\address{Institute of Mathematics \\ University of the National Education Commission, Krakow \\ Podchor\k{a}\.zych str. 2, 30-084 Krak\'ow, Poland}
\email{pawel.pasteczka@uken.krakow.pl}

\subjclass[2010]{26E60; 39B12; 39B22}
\keywords{functional equations; solvability; means; invariant means; uniqueness; Beta-type means}
\maketitle
\newcommand{\embed}{\triangleleft}
\newcommand{\om}{\prec}
\newcommand{\embedR}{\triangleright}
\newcommand{\omR}{\succ}

\section{Introduction}
An $n$-variable mean on an interval $I \subset \R$ is a function $M\colon I^n \to I$ which satisfies so-called mean-property, that is
$\min(v)\le M(v)\le \max(v)$ for all $v \in I^n$.
We say that a mean is symmetric if for every $v \in I^n$ and permutation $\sigma$ of $\{1,\dots,n\}$ we have $M(v\circ\sigma)=M(v)$. 
From now on, let  $\MM_n(I)$ be the family of all symmetric, continuous $n$-variable means on $I$. 
In this paper, we focus only on means which are symmetric and continuous. 

There are a number of problems related to means. One of the most classical ones arises from the equality 
\Eq{EInv}{
K(M_1(v),\dots,M_n(v))=K(v) \qquad (v \in I^n),
}
where $K$ and $M_1,\dots,M_n$ are $n$-variable means on $I$, which we also denote in a brief form as $K \circ (M_1,\dots,M_n)=K$. 

In the most classical approach for a given sequence $(M_1,\dots,M_n)$ we are searching for the mean $K$ that satisfies \eq{EInv} (see, for example, \cite{BorBor87,JarJar18} and references therein; we describe it briefly in section~\ref{sec:ApplInvMean}). This is not the field we are going to focus on.   

The problem we are going to solve arises from the paper by Matkowski \cite{Mat99a} who posted in some sense the opposite question. Namely, he solves \eq{EInv} in the case $n=2$, and fixed means $K$ and $M_1$ ($M_2$ is the unknown mean). The main outcome of the paper \cite{Mat99a} states that this problem has a unique solution whenever $K$ is a continuous and strictly increasing mean (that is $K$ is strictly increasing in each of its variables). It is not easy to generalize this statement to the case $n>2$. We have at least a few approaches to this problem. For example, we assume that $K$ and all but one $M_i$-s are given, and we are searching for the missing $M_i$ (this approach is, however, not of big interest till now).

In the second approach we assume that $K$ is $(M_1,\dots,M_n)$-invariant (that is $K$ solves \eq{EInv}) but we unify some of $M_i$-s, that is, we search for a mean $F$ which solves the equation of the following type
\Eq{E148b}{
K(v)=K(M_1(v),\dots,M_m(v),\uel{(n-m)}{F(v)}))\qquad (v \in I^n),
}
where $m<n$ and all means are $n$-variable means on $I$. 
Let us observe that the length of the suffix containing $F(v)$ (the value $n-m$) is imposed by the domain of $K$, so we can omit it whenever convenient. 

Binding equalities \eq{EInv} and \eq{E148b} leads us the following functional equation
\Eq{*}{
K(M_1(v),\dots,M_n(v))=K(M_1(v),\dots,M_m(v),F(v),\dots,F(v)) \qquad (v \in I^n).
}
In this setting, we assume that $K$ is symmetric and replace a suffix of $M_i$-s by a single (unknown) mean $F$. Since we have already applied the invariance property, this equation is in fact of the form 
\Eq{E148}{
\mu(M_1(v),\dots,M_n(v))=\mu(M_1(v),\dots,M_m(v),F(v),\dots,F(v)) \qquad (v \in I^k),
}
where $m<n$, all $M_i$-s are $k$-variable means on $I$, and $\mu$ is an $n$-variable mean on $I$.
In this setup, we have an additional parameter $k\in\N$ that does not appear in \eq{EInv} and \eq{E148b}, since that approaches force $k=n$. Indeed, in this setting the length of the vector $v$ does not have to coincide with the number of means. Thus, at this stage, we replaced the previous notation of external mean ($K$ by $\mu$) to emphasize that, this time, it could have a different number of variables than $M_i$-s (and $F$).
We can also consider the same type of equality when $\mu$ is not symmetric, but it is somewhat more difficult to express it in a compact way. Perhaps the most comprehensive study of this problem in the nonsymmetric setup was presented in \cite{MatPas20b}. In that paper, this problem was solved under the additional assumption that $K$ (or $\mu$) satisfies \eq{EInv}, which simplifies the equality to the original formulation \eq{E148b}. However, this can be easily relaxed when the considered problem is stated in the form \eq{E148}.

In this note we are going to solve the pexiderized version of \eq{E148}, that is the equality of the form 
\Eq{MN}{
\mu(M_1(v),\dots,M_n(v))=\mu(S_1(v),\dots,S_m(v),F(v),\dots,F(v))\qquad (v \in I^k),
}
where $k,m,n\in\N$ with $m<n$ are parameters, $\mu\colon I^n \to \R$ is a continuous function which is strictly monotone in each of its variables, and $S_1,\dots,S_m,M_1,\dots,M_n$ are $k$-variable means on an interval $I$; $F\colon I^k\to I$ is the unknown function. 

Let us stress, that it could happen that equality \eq{MN} has no solution in the family of means (or even functions) $F\colon I^k \to I$. The aim of this paper is to solve \eq{MN}. More precisely, we study mutual relations between $\mu$, $M_i$-s, and $S_j$-s which 
ensure us that \eq{MN} has a unique solution $F$, which is a mean.
The key tool to solve this equation will be the new definition, so-called embeddability. In some sense, this is what this note is about.

\section{Embeddability}
We are going to introduce embeddability in three steps. First, we define it for vectors of reals (sect.~\ref{sec:EV}), then for functions (sect.~\ref{sec:EF}). Finally, in the next section (sect.~\ref{sec:EM}), we restrict our considerations to the mean setting.

\subsection{\label{sec:EV} Embeddability of vectors}

For $n \in \N$ and a vector $v \in \R^n$, we denote be $v^\uparrow$ (resp. $v^\downarrow$) the (uniquely determined) nondecreasing (resp. nonincreasing) permutation of elements in $v$.

For $m,n \in \N$ with $m \le n$ we say that a vector $v \in \R^m$ is \emph{ordered minorized} (resp. \emph{ordered majorized}) by a vector $w\in\R^n$ if $v^\uparrow_{k} \ge w^\uparrow_k$ (resp. $v^\downarrow_k\le w^\downarrow_{k}$) for all $k \in \{1,\dots,m\}$; we denote it by $v\omR w$ (resp. $v \om w$).
Finally, we say that $v\in\R^m$ is \emph{embedded} in $w \in\R^n$ if it is both ordered majorized and minorized by $w$; we denote it by $v \embed w$. In the next lemma, we show how these properties simplify in the case $m=n$. 

\begin{lem}\label{lem:m=n}
Let $n\in \N$ and $I\subset \R$ be an interval. Then
\begin{enumerate}[{\rm (a)}]
    \item\label{lem:m=n.1} $v\in I^n$ is order majorized by $w \in I^n$ if, and only if, $w$ is ordered minorized by $v$;
    \item\label{lem:m=n.2} order majorization (and minorization) restricted to $\R^n$ is reflexive and transitive;
    \item\label{lem:m=n.3} $v\in I^n$ is embedded in $w \in I^n$ if, and only if, $v$ is a permutation of $w$;
    \item\label{lem:m=n.4} for every continuous symmetric function $f \colon I^n \to \R$ which is nondecreasing in each of its variables, and every $v,w \in I^n$ with $v \prec w$ we have $f(v)\le f(w)$.
\end{enumerate}
\end{lem}

Proofs of all the above properties are quite straightforward, and thus we decide to omit them.
Due to this lemma, we can use both notations $w\om v$ and $v \omR w$ for order majorizations and minorizations, as it is either equivalent or has a disjoint domain. Thus, purely formally, \begin{equation*}\begin{aligned}
v \prec w &:\iff \begin{cases}
v^\downarrow_k\le w^\downarrow_{k} &\text{ for }k \in \{1,\dots,m\}; \ m \le n;\\
v^\uparrow_{k}\le w^\uparrow_k &\text{ for }k \in \{1,\dots,n\}; \ m > n,
\end{cases}\\
v \succ w &:\iff \begin{cases}
v^\uparrow_k\ge w^\uparrow_{k} &\text{ for }k \in \{1,\dots,m\}; \ m \le n;\\
v^\downarrow_k\ge w^\downarrow_{k} &\text{ for }k \in \{1,\dots,n\}; \ m > n;\\
\end{cases}
\end{aligned} 
\qquad\qquad\qquad (v \in \R^m; w \in \R^n).
\end{equation*}
Before proceeding further, let us show a few examples of majorization, minorization, and embeddability.
\begin{exa}\phantom{a}\\
(a) Let $v=(3,15)$, $w=(5,0,10)$. Then $v^{\uparrow}=(3,15)$ and  $w^\uparrow=(0,5,10)$. Whence $w^\uparrow_k \le v^\uparrow_k$ for $k\in\{1,2\}$, and $v \succ w$. On the other hand $v^{\downarrow}=(15,3)$ and  $w^\downarrow=(10,5,0)$. Thus $v^\downarrow_1 > w^\downarrow_1$ which shows that $v$ is not ordered majorized by $w$.\\
(b) Let $v=(3,8)$, $w=(5,0,10)$. Then, similarly to the previous case, we have $v \succ w$. This time, however, $v^{\downarrow}=(8,3)$ and  $w^\downarrow=(10,5,0)$ and $v^\downarrow_i \le w^\downarrow_i$ for $i \in\{1,2\}$ which shows that $v$ is ordered majorized by $w$. Therefore, $v$ is embedded in $w$ (briefly $v\embed w$). \\
(c) Let $v=(5,6,7)$, $w=(2,4,6,8)$. Then $v \omR w$, however $w^\downarrow_3=4< 5=v^\downarrow_3$, that is, $v \not \om w$.
\end{exa}

Now we show that these properties under the transformation by monotone functions either remain unchanged or reverse. To this end,
for a function $f\colon X \to Y$ and a vector $v\in X^n$, let us denote $\vec f(v):=(f(v_1),\dots,f(v_n))\in Y^n$.

\begin{lem}\label{lem:monotone}
Let $I \subset \R$ be an interval, $v,\ w$ be two vectors having entries in $I$, and $f \colon I\to \R$.
\begin{enumerate}[(a)]
    \item\label{lem:monotone.a} If $f$ is nondecreasing and $v \prec w$ then $\vec f(v) \prec \vec f(w)$.
    \item\label{lem:monotone.b} If $f$ is nonincreasing and $v \prec w$ then $\vec f(w) \prec \vec f(v)$.
    \item\label{lem:monotone.c} If $f$ is monotone and $v \embed w$ then $\vec f(v) \embed \vec f(w)$.
\end{enumerate}
\end{lem}
\begin{proof}
Fix $m,n\in\N$, $v \in \R^m$, and $w \in \R^n$.
    If $f$ is nondecreasing then we have $(\vec f(v))^\uparrow=\vec f(v^\uparrow)$; $(\vec f(w))^\uparrow=\vec f(w^\uparrow)$. Therefore, for $m\le n$, we have (here $\bigwedge$ is the generalized "and" operator)
    \Eq{*}{
    v \prec w 
    &\Longrightarrow \bigwedge_{k=1}^{m} v^\downarrow_k\le w^\downarrow_{k}
    \Longrightarrow \bigwedge_{k=1}^{m} f(v^\downarrow_k)\le f(w^\downarrow_{k})
    \Longrightarrow \bigwedge_{k=1}^{m} \big(\vec f(v^\downarrow)\big)_k\le \big(\vec f(w^\downarrow)\big)_k\\
    &\Longrightarrow \bigwedge_{k=1}^{m} \big(\vec f(v)\big)^\downarrow_k\le \big(\vec f(w)\big)^\downarrow_k
    \Longrightarrow \vec f(v) \prec \vec f (w).
    }
Similarly, for $m>n$,
    \Eq{*}{
    v \prec w 
    &\Longrightarrow \bigwedge_{k=1}^{n} v^\uparrow_k\le w^\uparrow_{k}
    \Longrightarrow \bigwedge_{k=1}^{n} f(v^\uparrow_k)\le f(w^\uparrow_{k})
    \Longrightarrow \bigwedge_{k=1}^{n} \big(\vec f(v^\uparrow)\big)_k\le \big(\vec f(w^\uparrow)\big)_k\\
    &\Longrightarrow \bigwedge_{k=1}^{n} \big(\vec f(v)\big)^\uparrow_k\le \big(\vec f(w)\big)^\uparrow_k
    \Longrightarrow \vec f(v) \prec \vec f (w),
    }
    which completes the first assertion. The proof of the second one is analogous. Namely, if $f$~is nonincreasing then
    \Eq{*}{
    v \prec w 
    &\Longrightarrow \bigwedge_{k=1}^{m} v^\downarrow_k\le w^\downarrow_{k}
    \Longrightarrow \bigwedge_{k=1}^{m} f(v^\downarrow_k)\ge f(w^\downarrow_{k})
    \Longrightarrow \bigwedge_{k=1}^{m} \big(\vec f(v^\downarrow)\big)_k\ge \big(\vec f(w^\downarrow)\big)_k\\
    &\Longrightarrow \bigwedge_{k=1}^{m} \big(\vec f(v)\big)^\uparrow_k\ge \big(\vec f(w)\big)^\uparrow_k
    \Longrightarrow \vec f(v) \succ \vec f (w) \qquad\qquad &(m\le n);\\
    v \prec w 
    &\Longrightarrow \bigwedge_{k=1}^{n} v^\uparrow_k\le w^\uparrow_{k}
    \Longrightarrow \bigwedge_{k=1}^{n} f(v^\uparrow_k)\ge f(w^\uparrow_{k})
    \Longrightarrow \bigwedge_{k=1}^{n} \big(\vec f(v^\uparrow)\big)_k\ge \big(\vec f(w^\uparrow)\big)_k\\
    &\Longrightarrow \bigwedge_{k=1}^{n} \big(\vec f(v)\big)^\downarrow_k\ge \big(\vec f(w)\big)^\downarrow_k
    \Longrightarrow \vec f(v) \succ \vec f (w) \qquad\qquad &(m>n).
    }

    To show the last implication, assume that $f$ is monotone and $v \embed w$, which means that $m \le n$, $v \prec w$, and $w \prec v$. Then (depending on the monotonicity of $f$) we use one of the first two parts to show that $\vec f(v) \prec \vec f(w)$ and $\vec f(w) \prec \vec f(v)$. Then, by the definition of embeddability, we get $\vec f(v) \embed \vec f(w)$.
\end{proof}

Now we are heading towards the solution of \eq{MN}. To this end, for an interval $I \subset \R$ and $n \in \N$, let $\CS_n(I)$ be the family of all continuous, symmetric functions $\mu \colon I^n \to \R$, which are strictly increasing in each of its variables. 

\begin{lem}\label{lem:Is}
Let $I$ be an interval, $n \in \N$, and $\mu \in \CS_n(I)$. Then for every $w \in I^n$, $m < n$, and $v \in I^m$ with $v \embed w$, the equation
\Eq{E22}{
\mu(v_1,\dots,v_m,\uel{n-m}{x})=\mu(w_1,\dots,w_n)
}
has the unique solution $x_0 \in I$.  
Moreover $\min(w)\le x_0 \le \max(w)$.
\end{lem}
\begin{proof}
For the sake of brevity,  define the function $f \colon I \to \R$ by
\Eq{*}{
f(x):=\mu(v_1,\dots,v_m,x,\dots,x).
}
Since $\mu$ is continuous and strictly increasing in each of its variables we obtain that so is $f$. Moreover, in this new setup, equation \eq{E22} becomes $f(x)=\mu(w)$.

However, since $v \embed w$, we know that $v_k^\downarrow \le w_k^\downarrow$ and $w_k^\uparrow \le v_k^\uparrow$ ($k\in\{1,\dots,m\}$) whence
\Eq{*}{
\alpha:=(\uel{n-m}{\min(w)},v_1,\dots,v_m)\prec w \prec (v_1,\dots,v_m,\uel{n-m}{\max(w)})=:\beta.
}

Applying the symmetry and monotonicity of $\mu$ again, by Lemma~\dotref{lem:m=n}{4}, we get 
\Eq{*}{
f(\min(w))= \mu(\alpha)\le \mu(w)\le \mu(\beta)=f(\max(w)).
}
Since $f$ is continuous and strictly increasing, there exists $x_0 \in [\min (w), \max (w)]$ such that $f(x_0)=\mu(w)$. To complete the proof, note that the equation $f(x)=\mu(w)$ has at most one solution in $I$. 
\end{proof}

\subsection{\label{sec:EF} Embeddability of functions}
For all $m,n \in \N$ with $m \le n$ and a set $X$, we say that a sequence of functions $f=(f_1,\dots,f_m)$, $f_i \colon X \to \R$ is \emph{order majorized} by  $g=(g_1,\dots,g_n)$, $g_j \colon X \to \R$ provided
\Eq{*}{
(f_1(x),\dots,f_m(x)) \om (g_1(x),\dots,g_n(x)) \text{ for all }x \in X.
} 
Analogously we introduce order minorization and embeddability of functions. We also adapt the same notations, that is: $f \om g$, $f\omR g$, and $f \embed g$, respectively. 

\begin{lem}[Implicit function theorem]\label{lem:IFT}
    Let $X$ be a metric space, $I \subset \R$ be an interval, $m,n \in \N$ with $m < n$, $\mu \in \CS_n(I)$, and $f_1,\dots,f_m,g_1,\dots,g_n \colon X \to I$ be such that $(f_1,\dots,f_m) \embed (g_1,\dots,g_n)$. 
    Then the functional equation 
    \Eq{219}{
    \mu\big(f_1(x),\dots,f_m(x),\uel{n-m}{\alpha(x)}\big)
    =\mu\big(g_1(x),\dots,g_n(x)\big)
    }
    has the unique solution $\alpha \colon X \to I$, and 
    \Eq{244}{
    \min(g_1,\dots,g_n)\le \alpha\le \max(g_1,\dots,g_n).
    }
    Moreover if $\mu,f_1,\dots,f_m$ and $g_1,\dots,g_n$ are continuous, then so is $\alpha$.
\end{lem}
\begin{proof}
By Lemma~\ref{lem:Is}, we know that for each $x \in X$ there exists a unique $\alpha(x)\in I$ which solves \eq{219}. Furthermore, by the same lemma, we know that $\alpha$ satisfies \eq{244}. Therefore, the whole effort in the proof is to verify the "moreover" part. 

Assume that $\mu,f_1,\dots,f_m$ and $g_1,\dots,g_n$ are continuous.
Fix $x_0 \in X$ and take any sequence $(x_k)_{k=1}^\infty$ of points in $X$ converging to $x_0$. Let $u \in \R$ be an arbitrary accumulation point of the sequence $(\alpha(x_n))_{n=1}^\infty$. 
Then there exists a subsequence $(n_k)$ such that $(\alpha(x_{n_k}))_{k=1}^\infty$ converges to $u$. By \eq{244} we know that $u \in I$. Moreover, if we substitute $x:=x_{n_k}$ to \eq{219} and take the limit $k \to \infty$, we obtain 
\Eq{*}{
    \mu\big(f_1(x_0),\dots,f_m(x_0),\uel{n-m}u\big)
    =\mu\big(g_1(x_0),\dots,g_n(x_0)\big),
    }
whence $u=\alpha(x_0)$. Since $u$ was an arbitrary accumulation point of the sequence $(\alpha(x_n))_{n=1}^\infty$, we find that it converges to $\alpha(x_0)$ which shows that $\alpha$ is continuous.
\end{proof}

\section{\label{sec:EM} Embedability of means}
Recall that means are simply functions that admit the additional property. Consequently, one can speak about the embeddability of means exactly like it was done in the case of functions. 
Nevertheless, we introduce a few handy notations which allow us to express our statements in a more compact way. For a sequence $M=(M_1,\dots,M_n)\in\MM_k(I)^n$ let us set 
\Eq{*}{
\EE_m(M)&:=\{S \in \MM_k(I)^m \colon S \embed M\}\qquad (m \in\{1,\dots,n-1\});\\
\EE(M)&:=\bigcup_{m=1}^{n-1} \EE_m(M).
}
Furthermore, for the sake of brevity, we define $\abs{S}$ as the number of means in $S$, that is, $\abs{S}=m$ for all $S \in \MM_k(I)^m$ ($m \in\{1,\dots,n-1\}$).

Based on Lemma~\ref{lem:IFT}, for $n,k \in \N$, an interval $I \subset \R$, $M \in \MM_k(I)^n$, $S \in \EE(M)$, and $\mu\in \CS_n(I)$ we define ${\mathscr T}_{S,M}(\mu) \colon I^k \to I$ so that for all $v \in I^k$ the value ${\mathscr T}_{S,M}(\mu)(v)$ is the solution $x$ of equation
    \Eq{*}{
    \mu(S_1(v),\dots,S_{\abs S}(v),\uel{(n-|S|)}x)=\mu(M_1(v),\dots,M_n(v)).
    }

First, we show that ${\mathscr T}_{S,M}(\mu)$ is a symmetric mean.
\begin{thm}\label{thm:T}
    Let $n,k \in \N$, $I \subset \R$ be an interval,  $\mu\in \CS_n(I)$, $M\in \MM_k(I)^n$, and $S \in \EE(M)$. Then ${\mathscr T}_{S,M}(\mu)\in \MM_k(I)$.
\end{thm}

\begin{proof}
Denote briefly $K:={\mathscr T}_{S,M}(\mu)$, $m:=|S|$. By Lemma~\ref{lem:IFT} we know that 
\Eq{*}{
\min(M_1,\dots,M_n)\le K \le \max (M_1,\dots,M_n),
}
therefore $K$ is an $n$-variable mean on $I$. The "moreover" part of this lemma implies that $K$ is also continuous. 
Finally, since all $S_i$-s and $M_j$-s are symmetric, for every permutation $\tilde v$ of a vector $v \in I^k$, we have
\Eq{*}{
   \mu(S_1(v),\dots,S_m(v),\uel{}{K(v)})&=\mu(M_1(v),\dots,M_n(v)) \\
   &=\mu(M_1(\tilde v),\dots,M_n(\tilde v))\\
&= \mu(S_1(\tilde v),\dots,S_m(\tilde v),\uel{}{K(\tilde v)})\\
&= \mu(S_1(v),\dots,S_m(v),\uel{}{K(\tilde v)}).
}
Since $\mu\in \CS_n(I)$, this equality yields $K(\tilde v)=K(v)$, which implies that $K$ is symmetric. 
\end{proof}

The next lemma is a sort of comparability-type statement. More precisely we show that $ {\mathscr T}_{S,M}$ is monotone with respect to $S$ and $M$ (in the majorization ordering).

\begin{lem}\label{lem:Membed}
        Let $n,k \in \N$ and $I \subset \R$ be an interval. Take $M,M^* \in \MM_k(I)^n$ with $M^* \prec M$ and $S\in \EE(M)$, $S^* \in \EE(M^*)$ with $S \prec S^*$ and $\abs{S}=\abs{S^*}$. Then 
    \Eq{*}{
    {\mathscr T}_{S^*,M^*}(\mu) \le {\mathscr T}_{S,M}(\mu)\quad\text{ for all }\mu\in \CS_n(I).
    }
\end{lem}

\begin{proof}
Fix $\mu\in \CS_n(I)$, $v \in I^k$, $m:=|S|$ and set $a:={\mathscr T}_{S,M}(\mu)(v)$;  $a^*:={\mathscr T}_{S^*,M^*}(\mu)(v)$. By Lemma~\dotref{lem:m=n}{4}, we have
\Eq{*}{
\mu(M_1(v),\dots,M_n(v)) &\ge \mu(M_1^*(v),\dots,M_n^*(v)).
}
Moreover, for each $x \in I$, the map 
\Eq{*}{
I^m \ni (y_1,\dots,y_m)\mapsto \mu(y_1,\dots,y_m,\uel{n-m}x) \in I
}
is nondecreasing in each variable. Whence, applying Lemma~\dotref{lem:m=n}{4} again, for all $x \in  I$ we get
\Eq{*}{
\mu(S_1(v),\dots,S_m(v),\uel{}x) &\le \mu(S_1^*(v),\dots,S_m^*(v),\uel{}x)\qquad(x\in I).}
Substracting the inequalities above side-by-side, we get $f(x)\le f^*(x)$, where
\Eq{*}{
f(x)&:= \mu(S_1(v),\dots,S_m(v),\uel{}x) - \mu(M_1(v),\dots,M_n(v)) &\quad(x \in I); \\
f^*(x)&:= \mu(S_1^*(v),\dots,S_m^*(v),\uel{}x)- \mu(M_1^*(v),\dots,M_n^*(v)) &\quad(x \in I).
}
Since $\mu$ is continuous and strictly increasing in each of its variables, we find that so are $f$ and $f^*$. Furthermore, $f(a)=f^*(a^*)=0$. 
Thus $f(a^*) \le f^*(a^*)=0=f(a)$ which implies $a^* \le a$.
\end{proof}

\section{Applications}

This section contains three parts. First, we focus exclusively on the family of power means. The aim of this subsection is to show a typical application of Lemma~\ref{lem:monotone} and Theorem~\ref{thm:T}. In the second subsection, we present the important subcase of Theorem~\ref{thm:T} (${|S|}=1$). It leads us to the generalization of beta-type means (cf. Himmel-Matkowski \cite{HimMat18a}), which we are going to describe in the next step. In the final subsection we go back to the notion of invariance to generalize the results contained in Matkowski-Pasteczka \cite{MatPas20b}.

\subsection{Power means}
In this section, we study the embeddability of two maps consisting of power means. Recall that the $n$-variable power mean of order $s$ is defined by
\Eq{*}{
\P_s(x_1,\dots,x_n)=\begin{cases}
\Big(\dfrac{x_1^s+\cdots+x_n^s}{n}\Big)^{1/s} &\quad \text{ if } s\in\R\setminus\{0\}, \\[2mm]                                              
\sqrt[n]{x_1\cdots x_n} &\quad \text{ if } s = 0,
\end{cases}
}
where $n \in \N$ and $x_1,\dots,x_n \in \R_+$. Now we deliver the necessary and sufficient conditions for the embeddability of two sequences containing power means only. Remarkably, this result is based on the classical fact stating that power means are nondecreasing in their parameter.
Therefore, the result below can be easily adapted to other families of means.

\begin{prop}\label{prop:PM}
    Let $m, n \in \N$ $m \le n$ and $\alpha \in \R^m$, $\beta \in \R^n$. Then $\alpha \embed \beta$ if and only if 
    $(\P_{\alpha_1},\dots,\P_{\alpha_m})\embed (\P_{\beta_1},\dots,\P_{\beta_n})$.
\end{prop}
\begin{proof}
For a vector $x$ of positive numbers, define $f_x\colon \R \to \R_+$ by $f_x(s):=\P_s(x)$. Then 
\Eq{*}{
(\P_{\alpha_1},\dots,\P_{\alpha_m})\embed (\P_{\beta_1},\dots,\P_{\beta_n}) \iff \Big( \vec{f_x}(\alpha) \embed \vec{f_x}(\beta) \text{ for all }x \in\bigcup_{n=1}^\infty \R_+^n\Big).
}

Observe first that if $x$ is a constant vector, then the property $\vec f_x(\alpha)\embed \vec{f_x}(\beta)$
is trivially valid. Otherwise, it is a classical result saying that for every nonconstant vector $x$ the mapping $f_x$ is strictly increasing and continuous (and so is its inverse $f_x^{-1}$). Thus, by Lemma~\dotref{lem:monotone}{c}, we obtain that $\vec f_x(\alpha)\embed \vec{f_x}(\beta)$ holds if, and only if, $\alpha \embed \beta$.
\end{proof}
\begin{exa}
Let $n=4$, $m=2$, $\alpha=(0,2)$, $\beta=(-2,-1,1,3)$. Then $\alpha \embed \beta$ and whence, by Proposition~\ref{prop:PM}, $(\P_0,\P_2)\embed (\P_{-2},\P_{-1},\P_1,\P_3)$. Then, in view of Theorem~\ref{thm:T}, we have $\mathscr{T}_{(\P_0,\P_2),(\P_{-2},\P_{-1},\P_1,\P_3)} \colon \CS_4(I) \to \MM_k(\R_+)$. If we now take $+_4\colon \R_+^4 \to \R$ and $*_4\colon \R_+^4 \to \R_+$ as a sum and product of four variables we obtain two means in $\MM_k(\R_+)$:
\Eq{*}{
\M_1(v):=\mathscr{T}_{(\P_0,\P_2),(\P_{-2},\P_{-1},\P_1,\P_3)}(+_4)(v)&=\frac{\P_{-2}(v)+\P_{-1}(v)+\P_1(v)+\P_3(v)-\P_0(v)-\P_2(v)}2;\\
\M_2(v):=\mathscr{T}_{(\P_0,\P_2),(\P_{-2},\P_{-1},\P_1,\P_3)}(*_4)(v)&=\sqrt{\frac{\P_{-2}(v)\P_{-1}(v)\P_1(v)\P_3(v)}{\P_0(v)\P_2(v)}}.
}
Moreover $\P_{-2}\le\M_i\le\P_3$ ($i\in\{1,2\}$).
\end{exa}
\subsection{Generalized beta-type means}

Following Himmel--Matkowski \cite{HimMat18a}, for a given $k\in \mathbb{N}$ we define a $k$-variable Beta-type mean $\B_k\colon \R_{+}^k\rightarrow \R_{+}$ by 
\begin{equation*}
\B_k(v_{1},\dots ,v_k):=\bigg(\frac{kv_{1}\cdots v_k}{%
v_{1}+\dots +v_k}\bigg)^{\frac{1}{k-1}}.
\end{equation*}%
This is a particular case of so-called biplanar-combinatoric means ({\it Media biplana combinatoria}) defined in Gini \cite{Gin38} and Gini--Zappa \cite{GinZap38}. We deliver another generalization of this mean.

Indeed, the following proposition is the immediate consequence of Theorem~\ref{thm:T} with $\alpha=1$ and $M_j$-s defined as a projection to $j$-th variable. 
\begin{prop}
        Let $k \in \N$ and $I \subset \R$ be an interval, $\mu\in \CS_k(I)$ and $S \in \MM_k(I)$. Then the function $S^{\{\mu\}}\colon I^k\to I$ defined as the unique solution $x$ of the equation
    \Eq{*}{
    \mu(S(v),\uel{k-1}{x})=\mu(v)\qquad (v \in I^k)
    }
    is a continuous, symmetric $k$-variable mean on $I$.
\end{prop}

\begin{rem} 
In the particular case when $\mu=\P_0$ and $S=\P_1$ we obtain that for all $k \in \N$ and $v_1,\dots,v_k \in \R_+$, the value of $S^{\{\mu\}}(v_1,\dots,v_k)$ is the solution $x$ of the equation
\Eq{*}{
\sqrt[k]{\frac{v_1+\dots+v_k}k x^{k-1}}=\sqrt[k]{v_1\dots v_k}.
}
After easy simplification we obtain 
\Eq{*}{
S^{\{\mu\}}(v_1,\dots,v_k)=\bigg(\frac{kv_{1}\cdots v_k}{%
v_{1}+\dots +v_k}\bigg)^{\frac{1}{k-1}}=\B_k(v_1,\dots,v_k),
}
which shows that $\B_k=S^{\{\mu\}}=\P_1^{\{\P_0\}}$.
\end{rem}

\subsection{\label{sec:ApplInvMean} Case of invariant mean}
In this short section, we generalize the notion of complementary means introduced recently in \cite{MatPas20b}. Recall that $K \colon I^n \to I$ is invariant with respect to the mean-type mapping $M \in \MM_n(I)^n$ (briefly \emph{$M$-invariant}) if it solves the functional equation \eq{EInv}. 
There are a few classical sufficient conditions to warranty that there is exactly one \mbox{$M$-invariant} mean. The most classical assumptions (see for example \cite[Theorem~8.7]{BorBor87}) claim that each $M_k$ is strict (that is $\min(v)<M_k(v)<\max(v)$ for every nonconstant vector $v \in I^n$) and continuous. This assumption can be relaxed (see for example \cite{Mat09e} or \cite{MatPas21}). In our setup, we assume that all means belong to $\MM_n(I)\cap \CS_n(I)$, which implies strictness and continuity. 

We now formulate Theorem~\ref{thm:T} in the case when $\mu$ is $M$-invariant mean
\begin{prop}
     Let $n \in \N$, $I \subset \R$ be an interval, $M\in (\MM_n(I)\cap\CS_n(I))^n$, and $S \in \EE(M)$. Moreover, let $K\in \CS_n(I)$ be the (unique) $M$-invariant mean. Then the functional equation 
      \Eq{*}{
    K(S_1(v),\dots,S_{\abs S}(v),\uel{(n-|S|)}{T(v)})=K(v)\qquad (v \in I^n)
    }
    possesses exactly one solution $T_0$ in the family of means. Moreover $T_0={\mathscr T}_{S,M}(K)\in \MM_n(I)$.
\end{prop}

This proposition improves the setup of the paper \cite{MatPas20b} where it was shown in the case when $S$ is a subsequence of $M$.

\end{document}